\documentclass[runningheads]{llncs}
\usepackage{times}

\usepackage{soul}
\usepackage{url}
\usepackage[hidelinks]{hyperref}
\usepackage[utf8]{inputenc}
\usepackage[small]{caption}
\usepackage{graphicx}
\usepackage{amsmath}
\usepackage{booktabs}
\usepackage{graphicx}
\usepackage{bm}
\usepackage{url}
\usepackage{amsfonts}
\usepackage{mathrsfs}
\usepackage{algorithm}
\usepackage{algorithmic}

\DeclareMathOperator*{\argmin}{arg\,min}
\begin{document}
\title{A generalised log-determinant regularizer for online semi-definite programming and its applications\thanks{This work was supported by JSPS KAKENHI Grant Numbers JP19H04174 and JP19H04067, respectively.}}
\titlerunning{A generalised log-determinant regularizer for online semi-definite programming}
%
\author{
Yaxiong Liu\inst{1,3}\and Ken-ichiro Moridomi\inst{4}
 \and Kohei Hatano\inst{2,3} \and  Eiji Takimoto\inst{1}\\
 \institute{Department of Informatics, Kyushu University, Japan\and Faculty of Arts and Science, Kyushu University, Japan \and RIKEN AIP, Japan\and SMN Corporation}
 \email{\{yaxiong.liu,hatano,eiji\}@inf.kyushu-u.ac.jp\inst{1}\inst{2}\inst{3} kenichiro\_moridomi@so-netmedia.jp \inst{4}}
}

\authorrunning{Yaxiong Liu\and Ken-ichiro Moridomi \and Kohei Hatano \and Eiji Takimoto}

\maketitle             

\begin{abstract}
We consider a variant of online semi-definite programming problem (OSDP): The decision space consists of semi-definite matrices with bounded $\bm{\Gamma}$-trace norm, which is a generalization of trace norm defined by a positive definite matrix $\bm{\Gamma}.$ To solve this problem, we utilise the follow-the-regularized-leader algorithm with a $\bm{\Gamma}$-dependent log-determinant regularizer. Then we apply our generalised setting and our proposed algorithm to online matrix completion(OMC) and online similarity prediction with side information. In particular, we reduce the online matrix completion problem to the generalised OSDP problem, and the side information is represented as the $\bm{\Gamma}$ matrix. Hence, due to our regret bound for the generalised OSDP, we obtain an optimal mistake bound for the OMC by removing the logarithmic factor.
\end{abstract}

\section{Introduction}
Online semi-definite programming(OSDP)\cite{hazan2012near} plays a central role in online learning with matrix. Usually, OSDP is given as follows: on round $t\in [T],$ algorithm predicts a matrix $\bm{W}_{t} \in \mathcal{K},$ then adversary gives a loss matrix $\bm{L}_{t},$ and the algorithm incurs the Frobenius inner product of $\bm{W}_{t}$ and $\bm{L}_{t}$ as $\bm{W}_{t} \bullet \bm{L}_{t}.$ Our goal is to minimize the regret defined as \begin{equation}
\mathrm{Regret}_{T}=\sum_{t=1}^{T}\bm{W}_{t} \bullet \bm{L}_{t}-\min_{\bm{W} \in \mathcal{K}}\sum_{t=1}^{T}\bm{W}\bullet \bm{L}_{t},
\end{equation}
where $\mathcal{K}$ is a set of positive definite matrices with bounded trace norm $\mathrm{Tr}(\bm{W}) \leq \tau, \bm{W} \in \mathcal{K}.$ To solve OSDP problem, follow the regularizer leader(FTRL), a traditional algorithm in online learning \cite{shalev2012online}, is always involved. For different online learning problem like expert advice, by choosing different regularizer, like entropy, we can obtain satisfying regret bound \cite{cesa2006prediction}. To OSDP, \cite{moridomi2018online} give FTRL with log-determinant regularizer and obtain regret bound as $O(\sqrt{\tau T}).$ This model has been wildly utilised in online collaborative filtering \cite{shamir2011collaborative} \cite{cesa2011efficient} \cite{koltchinskii2011nuclear}, online max-cut problem \cite{lasserre2016max}, and min-max problem \cite{laraki2012semidefinite}.

In this paper we consider a generalization of online semi-definite programming problem with bounded $\bm{\Gamma}$-trace norm, where $\bm{\Gamma}$ is a strictly positive definite matrix.  Hence, the constraint of $\bm{W}$ is presented as $\mathrm{Tr}(\bm{\Gamma W \Gamma}) \leq \tau,$ a bounded $\bm{\Gamma}$-trace norm. We believe that this variance recovers the usual form, if $\bm{\Gamma}$ is identity matrix. In our setting, algorithm from \cite{moridomi2018online} can not be directly applied, since $\mathrm{Tr}(\bm{W})$ can not be bounded by $\mathrm{Tr}(\bm{\Gamma W \Gamma}),$ say $\mathrm{Tr}(\bm{\Gamma W \Gamma})\leq \tau$ implies not $\mathrm{Tr}(\bm{W}) \leq \tau,$ while $\bm{\Gamma}$ is an arbitrary strictly positive definite matrix. Therefore, we generalise the log-determinant regularizer with respect to $\bm{\Gamma}$, and achieve an upper bound to regret as $O(\sqrt{\tau T})$ for our generalised OSDP problem.

We believe that our generalised setting is not castle in the air. In this paper we involve our result to online matrix completion(OMC) with side information \cite{herbster2020online} and online similarity prediction \cite{gentile2013online} with side information. We firstly show that OMC with side information can be reduced to our generalised OSDP problem. Instead of FTRL with log-determinant regularizer in \cite{moridomi2018online}, the bound is only related to the size of comparator matrix, if utilizing our proposed algorithm, we can obtain a even tighter mistake bound. In OMC problem, side information, associated with row and column, implies the ``predictiveness'' of comparator matrix by some inherent characters, and we can represent this side information in our generalised OSDP as matrix $\bm{\Gamma}.$ For an ideal case, if the comparator matrix is latent block structured, the \emph{quasi-dimension}, based on side information, can efficiently reduce the mistake bound. The reduction and algorithm for online similarity prediction is same as OMC with side information.

So in this paper our main contribution is as follows:
\begin{itemize}
\item 1. We extend the FTRL algorithm for the generalised OSDP problem with bounded $\bm{\Gamma}$-trace norm, by introducing a new log-determinant regularizer depending on the matrix $\bm{\Gamma}$ 
     and give a regret bound. Note that our result recovers the previously known bound \cite{moridomi2018online} in the case that $\bm{\Gamma}$ is the identity matrix.

\item 2. Applications of our first technical results contain the OMC \cite{herbster2016mistake,herbster2020online} and the online similarity prediction \cite{gentile2013online,herbster2020online} with side information. For the OMC with side information we firstly reduce the problem to the generalised OSDP where the side information is encoded as the matrix $\bm{\Gamma}.$ Then by running the proposed algorithm the FTRL with our $\bm{\Gamma}$ dependent log-determinant regularizer, we achieve an optimal mistake  bound of OMC, which matches the lower bound when $\bm{\Gamma}$ is the identity matrix(no side information case \cite{herbster2016mistake}), improving the previous result of Herbster et.al \cite{herbster2020online} by a logarithmic factor.
     Furthermore, we show our reduction and an algorithm for the online similarity prediction and obtain improved mistake bound without a logarithmic factor as well.
\end{itemize}

Our paper is composed as follows: In section 3, we give the formal setting of the OSDP and the OMC with side information. The main algorithm and regret bound are given in section 4. In section 5 and 6 we show application of our proposed algorithm to the OMC and online similarity prediction with side information. In appendix we describe technique lemmata, and some details of a case where side information matters.



\section{Related work}

\textbf{OSDP problem} has been explored by  \cite{hazan2012near} \cite{christiano2014online} \cite{moridomi2018online}. \cite{moridomi2018online} give a regret bound by running FTRL with log-determinant regularizer as
$\mathrm{Regret_{OSDP}} \leq O(\sqrt{\tau T}),$
where the decision set is a set of positive definite matrix with bounded trace $\mathrm{Tr}(\bm{W}) \leq \tau.$

\textbf{Online matrix completion} has been studied by \cite{herbster2016mistake} and with side information \cite{herbster2020online}. In recently work \cite{herbster2020online}, authors give a mistake bound as
$O(\frac{\widehat{\mathcal{D}}}{\gamma^{2}} \ln (m+n))$ for the realizable case, where $\widehat{\mathcal{D}}$ is lower bounded by quasi-dimension of the comparator matrix. In ideal case, if this comparator matrix obtains some latent structure, like $(k,l)$-biclustered, then we can set that the side information matrix as PD-Laplacian corresponding to comparator matrix, so $\widehat{\mathcal{D}}$ can achieve $O(k+l),$ which leads a tighter bound than the case that the side information is vacuous as $O(m+n).$

\textbf{Side information} is widely applied in online learning problem with graph-based information\cite{herbster2005online} and online similarity prediction\cite{gentile2013online}.
In these cases the side information matrices are given as the PD-Laplacian of the matrix corresponding the the graph. In this paper we left an additional section in Appendix B for this discussion. Furthermore, our reduction method can applied in Online community membership prediction as well in section 6.

\section{Preliminaries}
For a positive integer $N$, let $[N]$ denote the set $\{1,2,\dots,N\}$
and $\mathbb{S}^{N \times N}_{+}$ and $\mathbb{S}^{N \times N}_{++}$ denote
the set of $N \times N$ semi-positive symmetric matrices and
the set of $N \times N$ strictly positive symmetric positive definite matrices,
respectively. We define $\bm{E}$ as identity matrix.
For a matrix $\bm{X}$, let
$\bm{X}_i$ denote the $i$-th row of $\bm{X}.$ For $\bm{X} \in \mathbb{S}^{N \times N}_{++},$ we denote
$\mathrm{Tr}(\bm{X})=\sum_{i=1}^{N}|\lambda_{i}(\bm{X})|$ as trace norm of $\bm{X},$ further $\mathrm{Tr}(\bm{\Gamma X \Gamma})=\sum_{i=1}^{N}|\lambda_{i}(\bm{\Gamma X \Gamma})|$ as $\bm{\Gamma}$-trace norm for $\forall \bm{\Gamma} \in \mathbb{S}^{N \times N}_{++},$ where $\lambda_{i}(\bm{X})$ is the $i$-th largest eigenvalue of $\bm{X},$ and $\Vert \mathrm{vec}(\bm{X})\Vert_{p}=\left(\sum_{(i,j)}(\bm{X}_{i,j})^{p}\right)^{1/p}.$
We denote that the squared radius of $\bm{M} \in \mathbb{S}^{m \times m}_{+}$ as $\mathcal{R}_{\bm{M}}=\max_{i \in [m]}\bm{M}^{+}_{ii},$
where $\bm{M}^{+}$ is the pseudo inverse of matrix $\bm{M}.$
We define the class of $m \times d$ row-normalized matrices as $\mathcal{N}^{m,d}=\lbrace \bar{\bm{P}} \subset \mathbb{R}^{m \times d}: \Vert \bar{\bm{P}}_{i}\Vert_{2}=1, i \in [m]\rbrace.$

\subsection{Generalised online semi-definite programming with bounded $\bm{\Gamma}$-trace norm}

Our generalised online semi-definite problem$(\mathcal{K},\mathcal{L})$ with respect to bounded $\bm{\Gamma}$-trace norm is defined as follows:
Given a matrix $\bm{\Gamma} \in \mathbb{S}^{N \times N}_{++},$
we define $\mathcal{K}=\lbrace \bm{W} \in \mathbb{S}_{++}^{N \times N}: \mathrm{Tr}(\bm{\Gamma W\Gamma})\leq \tau, \forall i \in [N], |\bm{W}_{i,i}| \leq \beta \rbrace,$ as decision set, and
$\mathcal{L}=\lbrace \bm{L} \in \mathbb{S}_{+}^{N \times N}: \Vert \mathrm{vec}(\bm{L})\Vert_{1} \leq g \rbrace,$ as loss space, more precisely speaking a sparse loss space.
Thus our generalised OSDP problem is as follows: on round $t \in [T],$
\begin{itemize}
\item 1. Algorithm chooses a matrix $\bm{W}_{t} \in \mathcal{K},$
\item 2. Adversary gives a loss matrix $\bm{L}_{t} \in \mathcal{L},$
\item 3. Algorithm incurs the loss as $\bm{W}_{t} \bullet \bm{L}_{t}.$
\end{itemize}

Our goal is to upper bound following regret
\begin{equation}
\mathrm{Regret}_{T}=\sum_{t=1}^{T}\bm{W}_{t} \bullet \bm{L}_{t}-\min_{\bm{W} \in \mathcal{K}}\sum_{t=1}^{T}\bm{W}\bullet \bm{L}_{t}.
\end{equation}

Note that if $\bm{\Gamma}=\bm{E},$ then $\bm{\Gamma W\Gamma}=\bm{W},$ the original OSDP is as a special case of our setting.

\cite{moridomi2018online} introduce an algorithm follow the regularized leader(FTRL) with log-determinant regularizer in matrix form. In OSDP $(\mathcal{K}, \mathcal{L}),$ we give a specific regularizer $R: \mathcal{K}\rightarrow \mathbb{R}$ and choose a matrix $\bm{W}_{t} \in \mathcal{K}$ on each round $t$ according to
\begin{equation}\label{algorithm:FTRL}
\bm{W}_{t}=\argmin_{\bm{W} \in \mathcal{K}}\left(R(\bm{W})+\eta\sum_{s=1}^{t-1}\bm{L}_{s}\bullet \bm{W}\right).
\end{equation}

The log-determinant regularizer is defined as
\begin{equation}
R(\bm{X})=-\ln \det (\bm{X}+\epsilon \bm{E}),
\end{equation}
where $\epsilon$ is positive.

In \cite{moridomi2018online}, the regret bound of original OSDP is restricted in the case that $\bm{\Gamma}=\bm{E}$ and the bound is $O(g\sqrt{\tau \beta T}).$

\subsection{Online matrix completion(OMC) with side information}
Consider a binary matrix $\lbrace -1,+1\rbrace^{m \times n},$
on each round $t,$ adversary sends $(i_{t},j_{t}) \in [m]\times [n]$ to algorithm. Then algorithm predicts $\hat{y}_{t} \in \lbrace -1,+1\rbrace.$ Next adversary reveals $y_{t} \in \lbrace-1,1\rbrace$ to the algorithm, at last algorithm suffers the loss $l_{t} = \mathbb{I}_{y_{t} \neq
\hat{y}_{t}}$, where $\mathbb{I}_{\cdot}=1$ if the event $\cdot$ is true and
$0$, otherwise. The goal of the algorithm is
to minimize the total loss (i.e. the number of mistakes) $M=\sum_{t=1}^T
\mathbb{I}_{y_{t} \neq \hat{y}_{t}}$.

We define a hinge loss function $h_{\gamma}:\mathbb{R} \to \mathbb{R}$, as
\[
	h_{\gamma}(x) = \begin{cases}
		0 & \text{if $\gamma \leq x$,} \\
		1-x/\gamma & \text{otherwise},
	\end{cases}
\]
for $\gamma >0.$
Assume that a sequence $\mathcal{S}=((i_{1},j_{1}),
y_{1}),\cdots, ((i_{T},j_{T}), y_{T})\subseteq ([m]\times[n]\times \lbrace-1,1\rbrace)^{T},$ and
let $\bm{P}$ and $\bm{Q}$ be matrices such that
$\bm{P}\bm{Q}^{T} \in \mathbb{R}^{m \times n},$
then, we define the hinge loss of the sequence $\mathcal{S}$ with respect to
$(\bm{P},\bm{Q})$ and $\gamma$ as
\begin{equation}\label{equation:hloss2}
	\mathrm{hloss}(\mathcal{S},(\bm{P},\bm{Q}),\gamma) =
		\sum_{t=1}^T h_{\gamma}\left(
			\frac{y_t \bm{P}_{i_t} \bm{Q}_{j_t}^{T}}
			{\|\bm{P}_{i_t}\|_{2} \|\bm{Q}_{j_t}\|_{2}}
		\right).
\end{equation}

The max norm of a matrix $\bm{U} \in \mathbb{R}^{m \times n}$ is defined by
\begin{equation}
\Vert \bm{U} \Vert_{\max}=\min_{\bm{PQ}^{T}=\bm{U}} \lbrace \max_{1 \leq i \leq m}\Vert \bm{P}_{i}\Vert \max_{1 \leq j \leq n}\Vert \bm{Q}_{j} \Vert \rbrace,
\end{equation}
where the minimum is over all matrices $\bm{P} \in \mathbb{R}^{m \times d}$ and $\bm{Q}\in \mathbb{R}^{n \times d}$ for all $d.$ So we define
\emph{quasi-dimension} of matrix $\bm{U} \in \mathbb{R}^{m \times n}$ with respect to \emph{side information} $\bm{M} \in \mathbb{S}^{m \times m}_{++}$ and $\bm{N} \in \mathbb{S}^{n \times n}_{++}$ at margin $\gamma$ is defined as
\begin{equation}
\mathcal{D}_{\bm{M},\bm{N}}^{\gamma}(\bm{U})=\min_{\bar{\bm{P}}\bar{\bm{Q}}^{T}=\gamma \bm{U}}\mathcal{R}_{\bm{M}} \mathrm{Tr} \left(\bar{\bm{P}}^{T} \bm{M} \bar{\bm{P}}\right)+\mathcal{R}_{\bm{N}}\mathrm{Tr} \left( \bar{\bm{Q}}^{T}\bm{N}\bar{\bm{Q}}\right),
\end{equation}
where $\bar{\bm{P}} \in \mathcal{N}^{m,d}$ and $\bar{\bm{Q}} \in \mathcal{N}^{n,d}.$ Note that only when $\Vert \bm{U} \Vert_{\max} \leq 1/\gamma,$ the infimum exists. Here we define that $\bm{M},\bm{N}$ as \emph{side information}, especially if $\bm{M}$ and $\bm{N}$ are identity matrix, and $\mathcal{D}^{\gamma}_{\bm{M,N}}=m+n$ when side information is vacuous. In following part we simplifies this quasi-dimension as $\mathcal{D},$ when it leads to no ambiguity. Moreover in following part we denote $\widehat{\mathcal{D}} \geq \mathcal{D}_{\bm{M},\bm{N}}^{\gamma}(\bm{U})$ for a fixed \emph{comparator matrix} $\bm{U}$ to OMC problem with side information $\bm{M}, \bm{N}$ at margin $\gamma.$

Let $G=(V,E,W)$ be an $m$-vertex connected, weighted and undirected graph with positive weights. Let $\bm{A}$ be the $m \times m$ matrix such that $\bm{A}_{ij}=\bm{A}_{ji}=W_{ij}$ if $(i,j)\in E$ and $\bm{A}_{ij}=0,$ otherwise. Let $\bm{D}$ be a $m \times m$ diagonal matrix such that $\bm{D}_{ii}$ is the degree of each vertex $i.$ We define that $\bm{L}=\bm{D}-\bm{A}$ as \emph{Laplacian}. Furthermore positive definite Laplacian(\emph{PD-Laplacian}) is given as
$\bar{\bm{L}}=\bm{L}+\mathcal{R}_{\bm{L}}\frac{1}{m^{2}}\bm{I},$ where $\bm{I} \in \mathbb{R}^{m \times m}$ is a matrix, whose entries are all $1.$

\section{Algorithm for OSDP with bounded $\bm{\Gamma}$-trace norm and regret bound}
We utilise FTRL (\ref{algorithm:FTRL}) with generalised log-determinant regularizer (\ref{Equation:generalised-log-determinant-regularizer}), defined in follow, to our generalised OSDP problem with respect to bounded $\bm{\Gamma}$-trace norm $(\mathcal{K},\mathcal{L}),$ where
\begin{equation*}
\begin{split}
&\mathcal{K}=\left\{\bm{W} \in \mathbb{S}_{++}^{N \times N}: |\bm{W}_{ii}| \leq \beta, \mathrm{Tr}(\bm{\Gamma  W\Gamma}) \leq \tau\right\} \\
&\mathcal{L}=\left\{\bm{L} \in \mathbb{S_{+}}^{N \times N}: \Vert  \mathrm{vec}(\bm{L}) \Vert_{1}=\sum_{i,j}|\bm{L}_{i,j}| \leq g \right\},
\end{split}
\end{equation*}
for a fixed $\bm{\Gamma} \in \mathbb{S}^{N \times N}_{++}.$

We define \emph{generalised log-determinant regularizer} as follows:
\begin{equation}\label{Equation:generalised-log-determinant-regularizer}
R(\bm{X})= -\ln \det (\bm{\Gamma X \Gamma}+\epsilon \bm{E}).
\end{equation}

Next we give our regret bound for FTRL with generalised log-determinant regularizer in following theorem.
\begin{theorem}\label{theorem:upper-bound-of-OSDP}
Given $\bm{\Gamma} \in \mathbb{S}^{N \times N}_{++},$ and $\rho=\max_{i,j}|(\bm{\Gamma}^{-1}\bm{\Gamma}^{-1})_{i,j}|.$
Let
\[
	\mathcal{K} = \{ \bm{W} \in \mathbb{S}^{N \times N}_{++}:
			\Vert \mathrm{vec}(\bm{W}) \Vert_{\infty} \leq \beta,	\mathrm{Tr}(\bm{\Gamma W\Gamma})  \leq \tau
	\}
\]
for some $\beta >0$ and $\tau > 0,$ then let
$
	\mathcal{L} \subseteq \{ \bm{L} \in \mathbb{S}_{+}^{N \times N}:
		\sum_{i,j}|\bm{L}_{i,j}| \leq g
	\}
$
for some $g > 0$.
Then, for any competitor matrix $\bm{W}^{*} \in \mathcal{K}$,
the FTRL algorithm with respect to generalised log-determinant regularizer achieves
\[
	\mathrm{Regret}_\mathrm{OSDP}(T, \mathcal{K}, \mathcal{L}, \bm{W}^{*})
	= O\left(
		g^2(\beta + \rho\epsilon)^2 T\eta+\frac{\tau}{\epsilon \eta}
	\right).
\]
In particular, letting $\eta=\sqrt{\frac{\tau}{g^{2}(\beta+\rho\epsilon)^{2}\epsilon T}}$ we have
\begin{equation}
\mathrm{Regret}_{\mathrm{OSDP}}\leq O\left(\sqrt{g^{2}(\beta+\rho\epsilon)^{2}\tau T/\epsilon}\right).
\end{equation}
\end{theorem}

Note that in original OSDP, $\bm{\Gamma}=\bm{E},$ so the result $O(g\sqrt{\tau \beta T})$ in \cite{moridomi2018online} is a special case of our problem, by setting $\epsilon=\beta,$ and $\rho=1.$

Before we prove this theorem, we need to involve some Lemmata and notations.

The negative entropy function over the set of probability distribution $P$ over $\mathbb{R}^{N}$ is defined as $H(P)=\mathbb{E}_{x \sim P}[\ln (P(x))].$ The total variation distance between probability distribution $P$ and $Q$ over $\mathbb{R}^{N}$ is defined as $\frac{1}{2}\int_{x}|P(x)-Q(x)|dx.$ The characteristic function of a probability distribution $P$ over $\mathbb{R}^{N}$ is defined as $\phi(u)=\mathbb{E}_{x \sim P}[e^{i u^{T}x}]$ where $i$ is the imaginary unit.

\begin{definition}\label{definition:s-strongly-convex-over-KL}
For a decision space $\mathcal{K}$ and a real number $s \geq 0,$ a regularizer $R: \mathcal{K} \rightarrow \mathbb{R}$ is said to be $s$-strongly convex with respect to a loss space $\mathcal{L}$ if for any $\alpha \in [0,1]$ any $\bm{X},\bm{Y} \in \mathcal{K}$ and $\bm{L} \in \mathcal{L}:$
\begin{equation}
\begin{split}
R(\alpha \bm{X} +(1-\alpha)\bm{Y}) &\leq \alpha R(\bm{X})+(1-\alpha)R(\bm{Y})\\
&-\frac{s}{2}\alpha (1-\alpha)|\bm{L} \bullet (\bm{X}-\bm{Y})|^{2}.
\end{split}
\end{equation}
\end{definition}

\begin{lemma}
Let $G_{1}$ and $G_{2}$ are two zero mean Gaussian distribution with covariance matrix $\bm{\Gamma\Sigma \Gamma}$ and $\bm{\Gamma\Theta\Gamma}.$ Furthermore $\bm{\Sigma}$ and $\bm{\Theta}$ are positive definite matrices. If there exists $(i,j)$ such that
\begin{equation}
|\bm{\Sigma}_{i,j}-\bm{\Theta}_{i,j}|\geq \delta(\bm{\Sigma}_{i,i}+\bm{\Theta}_{i,i}+\bm{\Sigma}_{j,j}+\bm{\Theta}_{j,j}),
\end{equation}
then the total variation distance between $G_{1}$ and $G_{2}$ is at least $\frac{1}{12e^{1/4}}\delta.$
\end{lemma}

\begin{proof}
Given $\phi_{1}(u)$ and $\phi_{2}(u)$ as characteristic function of $G_{1}$ and $G_{2}$ respectively. Due to Lemma \ref{lemma:A-1} in \cite{moridomi2018online}, we have
\begin{equation}
\int_{x} |G_{1}(x)-G_{2}(x)| dx \geq \max_{u \in \mathbb{R}^{N}}|\phi_{1}(u)-\phi_{2}(u)|,
\end{equation}
So we only need to show the lower bound of $\max_{u \in \mathbb{R}^{N}}|\phi_{1}(u)-\phi_{2}(u)|.$

Then we set that characteristic function of $G_{1}$ and $G_{2}$ are $\phi_{1}(u)=e^{\frac{-1}{2}u^{T}\bm{\Gamma}^{T}\bm{\Sigma \Gamma} u}$ and $\phi_{2}(u)=e^{\frac{-1}{2}u^{T}\bm{\Gamma}^{T}\bm{\Theta} \bm{\Gamma} u}$ respectively. Setting that
$\alpha_{1}=(\bm{\Gamma} v)^{T}\bm{\Sigma} (\bm{\Gamma} v),$ $\alpha_{2}=(\bm{\Gamma} v)^{T}\bm{\Theta} (\bm{\Gamma} v)$ and $\bm{\Gamma} u=\frac{\bm{\Gamma} v}{\sqrt{\alpha_{1}+\alpha_{2}}}.$ Moreover we denote that $\bar{v}=\Gamma v,$ for any $\bar{v} \in \mathbb{R}^{V},$ there exists $v \in \mathbb{R}^{V}.$ $\bar{u}=\Gamma u$ in the same way.

We need only give the lower bound of $\max_{u \in \mathbb{R}^{N}} |\phi_{1}(u)-\phi_{2}(u)|.$

Next we have that
\begin{equation}
\begin{split}
&\max_{u \in \mathbb{R}^{N}} |\phi_{1}(u)-\phi_{2}(u)|\\
&=\max_{u \in \mathbb{R}^{N}}\left|e^{\frac{-1}{2} u^{T}\bm{\Gamma \Sigma \Gamma} u}-e^{\frac{-1}{2}u^{T}\bm{\Gamma\Theta\Gamma} u}\right| \\
&=\max_{u \in \mathbb{R}^{V}}\left|e^{\frac{-1}{2}(\bm{\Gamma} u)^{T}\bm{\Sigma} (\bm{\Gamma} u)}-e^{\frac{-1}{2}(\bm{\Gamma} u)^{T}\bm{\Theta}(\bm{\Gamma} u)}\right| \\
&\geq \max_{\bar{v} \in \mathbb{R}^{N}} \left| e^{\frac{-\alpha_{1}}{2(\alpha_{1}+\alpha_{2})}}-e^{\frac{-\alpha_{2}}{2(\alpha_{1}+\alpha_{2})}}\right| \\
& \geq \max_{\bar{v} \in \mathbb{R}^{N}} \left|\frac{1}{2e^{1/4}} \frac{\alpha_{1}-\alpha_{2}}{\alpha_{1}+\alpha_{2}}\right|.
\end{split}
\end{equation}
Then second inequality is due to Lemma \ref{lemma:A-4} in \cite{moridomi2018online}.

Due to assumption in the Lemma we obtain for some $(i,j)$ that
\begin{equation}
\begin{split}
&\delta(\bm{\Sigma}_{i,i}+\bm{\Theta}_{i,i}+\Sigma_{j,j}+\bm{\Theta}_{j,j}) \leq |\bm{\Sigma}_{i,j}-\bm{\Theta}_{i,j}| \\
&=\frac{1}{2}|(\bm{e}_{i}+\bm{e}_{j})^{T}(\bm{\Sigma-\Theta})(\bm{e}_{i}+\bm{e}_{j})\\
&-\bm{e}_{i}^{T}(\bm{\Sigma}-\bm{\Theta})\bm{e}_{i}-\bm{e}_{j}^{T}(\bm{\Sigma}-\bm{\Theta})\bm{e}_{j}|
\end{split}
\end{equation}
It implies that one of $(\bm{e}_{i}+\bm{e}_{j})^{T}(\bm{\Sigma}-\bm{\Theta})(\bm{e}_{i}+\bm{e}_{j}),$ $\bm{e}_{i}^{T}(\bm{\Sigma}-\bm{\Theta})\bm{e}_{i}$ and $\bm{e}_{j}^{T}(\bm{\Sigma}-\bm{\Theta})\bm{e}_{j}$ has absolute value greater that $\frac{2\delta}{3}(\bm{\Sigma}_{i,i}+\bm{\Theta}_{i,i}+\bm{\Sigma}_{j,j}+\bm{\Theta}_{j,j}).$

Since $\Sigma, \Theta$ are strictly positive definite matrices, we have that for all $v \in \lbrace \bm{e}_{i}+\bm{e}_{j},\bm{e}_{i},\bm{e}_{j}\rbrace$
\begin{equation}
v^{T}(\bm{\Sigma} +\bm{\Theta})v \leq 2(\bm{\Sigma}+\bm{\Theta})_{i,i}+(\bm{\Sigma}+\bm{\Theta})_{j,j}.
\end{equation}

and therefore we have that
\begin{equation}
\begin{split}
&\max_{\bar{v} \in \mathbb{R}^{N}} \left|\frac{1}{2e^{1/4}} \frac{\alpha_{1}-\alpha_{2}}{\alpha_{1}+\alpha_{2}}\right| \\
& \geq \max_{\bar{v} \in \lbrace \bm{e}_{i}+\bm{e}_{j},\bm{e}_{i},\bm{e}_{j}\rbrace}\left|\frac{1}{2e^{1/4}}\frac{v^{T}(\bm{\Sigma}-\bm{\Theta})v}{v^{T}(\bm{\Sigma}+\bm{\Theta})v}\right| \geq \frac{\delta}{6e^{1/4}}
\end{split}
\end{equation}
\qed
\end{proof}

\begin{lemma}{\label{lemma:primary-form-for-entropy}}
Let $X, Y \in \mathbb{S}_{+}^{N \times N}$ be such that
\begin{equation}
|\bm{X}_{i,j}-\bm{Y}_{i,j}|\geq \delta(\bm{X}_{i,i}+\bm{Y}_{i,i}+\bm{X}_{j,j}+\bm{Y}_{j,j}),
\end{equation}
and $\bm{\Gamma}$ is symmetric strictly positive definite matrix.
Then the following inequality holds that
\begin{equation}
\begin{split}
&-\ln \det(\alpha \bm{\Gamma X\Gamma} + (1-\alpha )\bm{\Gamma Y\Gamma}) \\
&\leq -\alpha\ln \det(\bm{\Gamma X\Gamma} )-(1-\alpha)\ln \det(\bm{\Gamma Y\Gamma} ) \\
&-\frac{\alpha(1-\alpha)}{2}\frac{\delta^{2}}{72e^{1/2}}.
\end{split}
\end{equation}
\end{lemma}

\begin{proof}
Let $G_{1}$ and $G_{2}$ are zero mean Gaussian distribution with covariance matrix $\bm{\Gamma \Sigma \Gamma} =\bm{\Gamma X\Gamma} $ and $\bm{\Gamma \Theta\Gamma} =\bm{\Gamma Y\Gamma} .$ In matrix total variation distance between $G_{1}$ and $G_{2}$ is at least $\frac{\delta}{12e^{1/4}}.$
We denote that $\tilde{\delta}=\frac{\delta}{12e^{1/4}}.$ Consider the entropy of the following probability distribution of $v$ with probability $\alpha$ that $v \sim G_{1}$ and $v \sim G_{2}$ otherwise. Its covariance matrix is $\alpha \bm{\Gamma \Sigma\Gamma} +(1-\alpha)\bm{\Gamma \Theta\Gamma} .$
Due to Lemma \ref{lemma:A-2} and \ref{lemma:A-3} in \cite{moridomi2018online} we obtain that
\begin{equation*}
\begin{split}
&\ln \det(\alpha \bm{\Gamma\Sigma \Gamma}+(1-\alpha)\bm{\Gamma \Theta\Gamma}) \\
&\leq 2H(\alpha G_{1}+(1-\alpha)G_{2})+ \ln (2\pi e)^{V} \\
& \leq 2\alpha H(G_{1})+2(1-\alpha)H(G_{2})+ \ln (2\pi e)^{V}-\alpha (1-\alpha)\tilde{\delta}^{2} \\
&=\alpha \ln \det(\bm{\Gamma \Sigma\Gamma})-(1-\alpha)\ln \det(\bm{\Gamma\Theta\Gamma})-\alpha (1-\alpha)\tilde{\delta}^{2}.
\end{split}
\end{equation*}
\qed
\end{proof}

\begin{proposition}
The generalised log-determinant regularizer $R(X)=-\ln \det(\bm{\Gamma X\Gamma} +\epsilon \bm{E})$ is $s$-strongly convex with respect to $\mathcal{L}$ for $\mathcal{K}$ with $s=1/(1152\sqrt{e}(\beta+\rho\epsilon)^{2}g^{2}).$ Here $\bm{E}$ is identity matrix.
\end{proposition}
\begin{proof}
Firstly we know that $\bm{\Gamma X\Gamma} +\epsilon \bm{E}=\bm{\Gamma}( \bm{X}+\bm{\Gamma}^{-1}\epsilon \bm{E \Gamma}^{-1})\bm{\Gamma}.$

Applying the Lemma \ref{lemma:lemma-for-strongly-convex} to $\bm{X}+\bm{\Gamma}^{-1}\epsilon \bm{E\Gamma}^{-1}$ and $\bm{Y} + \bm{\Gamma}^{-1}\epsilon \bm{E\Gamma}^{-1}$ for $\bm{X}, \bm{Y} \in \mathcal{K}$ where $\max_{i,j}|(\bm{X}+\bm{\Gamma}^{-1}\epsilon \bm{E\Gamma}^{-1})_{i,j}| \leq \max_{i,j}|\bm{X}_{i,j}|+\epsilon\rho,$ we have that $\beta^{'}=\beta +\epsilon\rho,$ where $\rho=\max_{i,j}|(\bm{\Gamma}^{-1}\bm{\Gamma}^{-1})_{i,j}|.$ According to Lemma \ref{lemma:primary-form-for-entropy} and Definition \ref{definition:s-strongly-convex-over-KL} we have this proposition.
\qed
\end{proof}

\begin{proof}[Proof of Theorem \ref{theorem:upper-bound-of-OSDP}]
Due to Lemma \ref{lemma:lemma-for-upper-bound-of-strongly-convex} we obtain that
\begin{equation}
\mathrm{Regret_{OSDP}}(T,\mathcal{K},\mathcal{L},\bm{W}^{*})  \leq \frac{H_{0}}{\eta}+\frac{\eta}{s}T.
\end{equation}
Due to above Proposition we know that $s=1/(1152(\beta+\rho\epsilon)^{2}\sqrt{e}g^{2}).$

Thus we need only to show $H_{0} \leq \frac{\tau}{\epsilon}.$
Given $\bm{W}_{0}$ and $\bm{W}_{1}$ is the minimizer and maximizer of $R$ respectively, then we obtain that
\begin{equation}
\begin{split}
&\max_{\bm{W},\bm{W}^{'} \in \mathcal{K}}(R(\bm{W})-R(\bm{W}^{'}))=R(\bm{W}_{1})-R(\bm{W}_{0})\\
&=-\ln \det(\bm{\Gamma W}_{1}\bm{\Gamma} +\epsilon \bm{E})+\ln \det(\bm{\Gamma W}_{0}\bm{\Gamma}+\epsilon \bm{E}) \\
&=\sum_{i=1}^{N}\ln \frac{\lambda_{i}(\bm{\Gamma W}_{0}\bm{\Gamma})+\epsilon}{\lambda_{i}(\bm{\Gamma W}_{1}\bm{\Gamma})+\epsilon}\\
&=\sum_{i=1}^{N}\ln \left(\frac{\lambda_{i}(\bm{\Gamma W}_{0}\bm{\Gamma} )}{\lambda_{i}(\bm{\Gamma W}_{1}\bm{\Gamma})+\epsilon}+\frac{\epsilon}{\lambda_{i}(\bm{\Gamma W}_{1}\bm{\Gamma})+\epsilon}\right)\\
& \leq \sum_{i=1}^{N}\ln \left(\frac{\lambda_{i}(\bm{\Gamma W}_{0}\bm{\Gamma})}{\epsilon}+1\right) \\
&\leq \sum_{i=1}^{N}\frac{\lambda_{i}(\bm{\Gamma W}_{0}\bm{\Gamma})}{\epsilon}=\frac{\mathrm{Tr} (\bm{\Gamma W}_{0}\bm{\Gamma})}{\epsilon} \leq \frac{\tau}{\epsilon}.
\end{split}
\end{equation}

So we have our conclusion.
\qed
\end{proof}

\section{Application to OMC with side information}
In this section, we show that the OMC with side information $\bm{M},\bm{N}$ can be reduced to our generalised OSDP with bounded $\bm{\Gamma}$-trace norm.
The reduction is twofold: Firstly reduce to
an online matrix prediction(OMP) problem with side information $\bm{M},\bm{N}$ and then further reduce to
our generalised OSDP problem. Meanwhile we utilise mistake-driven technique such that we can
bound the number of mistakes without regret bound with respect to $T.$

\subsection{Reduction from OMC to OMP with side information}
First we describe an OMP problem with side information $\bm{M}$ and $\bm{N}$, to which our problem is reduced.
The problem is specified by a competitor class
$\mathcal{X} \subseteq [-1,1]^{m \times n}$ defined as follows.
For any matrix $\bm{A} \in \mathrm{R}^{k \times l}$, we define
\begin{equation*}
	\bar{\bm{A}} = \mathrm{diag}\left(		\frac{1}{\Vert \bm{A}_{1} \Vert_{2}},\cdots, \frac{1}{\Vert \bm{A}_{k}\Vert_{2}}	\right)\bm{A}.
\end{equation*}
That is, $\bar{\bm{A}} \in \mathcal{N}^{k,l},$ is a matrix obtained from $\bm{A}$
by normalising all row vectors.
Then, our competitor class $\mathcal{X}$ is defined as
\begin{equation*}
	\mathcal{X} = \{		\bar{\bm{P}}\bar{\bm{Q}}^{T} : \bm{P}\bm{Q}^{T} \in \mathrm{R}^{m \times n} \land \mathcal{R}_{\bm{M}}\mathrm{Tr} \left(\bar{\bm{P}}^{T} \bm{M} \bar{\bm{P}}\right)+\mathcal{R}_{\bm{N}}\mathrm{Tr} \left( \bar{\bm{Q}}^{T}\bm{N}\bar{\bm{Q}}\right) \leq \widehat{\mathcal{D}}	\},
\end{equation*}
where $\widehat{\mathcal{D}} \geq \mathcal{D},$ and $\widehat{\mathcal{D}}$ is named as \emph{quasi dimension estimator}.

The OMP problem with side information for $\mathcal{X}$ is described as the following protocol.

For each round $t \in [T]$,
\begin{enumerate}
\item the algorithm chooses a matrix $\bm{X}_t \in \mathbb{R}^{m \times n}$,
\item observes a triple $(i_t,j_t,y_t) \in [m]\times[n]\times\{-1,1\}$,
	and then
\item suffers a loss given by $h_\gamma(y_t \bm{X}_{t,(i_t,j_t)})$.
\end{enumerate}
The goal of the algorithm is to minimize the regret:
\[
	\mathrm{Regret}_\mathrm{OMP}(T, \mathcal{X}, \bm{X}^*) =
	\sum_{t=1}^T h_\gamma(y_t \bm{X}_{t,(i_t,j_t)})
	- \sum_{t=1}^T h_\gamma(y_t \bm{X}^*_{(i_t,j_t)}),
\]
for any competitor matrix $\bm{X}^{*} \in \mathcal{X}$.
Note that unlike the standard setting of online prediction,
we do not require $\bm{X}_{t} \in \mathcal{X}$.

Below we give the reduction.
Assume that we have an algorithm $\mathcal{A}$ for the OMP problem.

Run the algorithm $\mathcal{A}$ and receive the first
prediction matrix $\bm{X}_1$ from $\mathcal{A}$.

In each round $t$,
\begin{enumerate}
\item observe an index pair $(i_t,j_t) \in [m] \times [n]$,
\item predict $\hat{y}_{t} = \mathrm{sgn}(\bm{X}_{t,(i_t,j_t)})$,
\item observe a true label $y_{t} \in \{-1,1\}$,
\item if $\hat y_t = y_t$ then $\bm{X}_{t+1} = \bm{X}_t$, and
if $\hat y_t \neq y_t$, then feed $(i_t,j_t,y_t)$ to $\mathcal{A}$ to let it
proceed and receive $\bm{X}_{t+1}$.
\end{enumerate}
Note that we run the algorithm $\mathcal{A}$ in the mistake-driven manner,
and hence $\mathcal{A}$ runs for
$M = \sum_{t=1}^T \mathbb{I}_{\hat{y}_{t} \neq y_{t}}$ rounds,     
where $M$ is the number of mistakes of the reduction algorithm above.

The next lemma shows the performance of the reduction.

\begin{lemma} \label{lem:reduction}
Let $\mathrm{Regret}_\mathrm{OMP}(M,\mathcal{X}, \bm{X}^*)$ denote the regret of
the algorithm $\mathcal{A}$ in the reduction above for a competitor
matrix $\bm{X}^* \in \mathcal{X}$, where
$M = \sum_{t=1}^T \mathbf{1}(\hat y_t \neq y_t)$. Then,
\begin{eqnarray}
M &\leq \inf_{\bm{P}\bm{Q}^{T} \in \mathcal{X}} (\mathrm{Regret}_\mathrm{OMP}(M, \mathcal{X}, \bar{\bm{P}}\bar{\bm{Q}}^{T})+ \mathrm{hloss}(\mathcal{S},(\bm{P},\bm{Q}),\gamma) ) \label{eq:bound1}\\
&\leq 	\sup_{\bm{X}^* \in \mathcal{X}} \mathrm{Regret}_\mathrm{OMP}(M, \mathcal{X}, \bm{X}^{*})+ \mathrm{hloss}(\mathcal{S}, \gamma), \label{eq:bound2}
\end{eqnarray}
where we define that
\begin{equation}
\mathrm{hloss}(\mathcal{S},\gamma)=\min_{\bar{\bm{P}}\bar{\bm{Q}}^{T} \in \mathcal{X}}\mathrm{hloss}(\mathcal{S},(\bm{P},\bm{Q}), \gamma).
\end{equation}
\end{lemma}

\begin{remark}\label{rem:original-case}
If $\bm{M}$ and $\bm{N}$ are identity matrices, then we have that $\mathcal{R}_{\bm{M}}\mathrm{Tr} \left(\bar{\bm{P}}^{T} \bm{M} \bar{\bm{P}}\right)+\mathcal{R}_{\bm{N}}\mathrm{Tr} \left( \bar{\bm{Q}}^{T}\bm{N}\bar{\bm{Q}}\right)=m+n.$ In this case $\mathcal{X}=\lbrace \bar{\bm{P}}\bar{\bm{Q}}^{T} : \bm{P}\bm{Q}^{T} \in \mathbb{R}^{m \times n} \rbrace,$ and $\widehat{\mathcal{D}}=m+n.$
\end{remark}

\begin{proof}
Let $\bm{P}$ and $\bm{Q}$ be arbitrary matrices such that
$\bm{P}\bm{Q}^{T} \in \mathbb{R}^{m \times n}$.
Since $\mathbf{1}(\mathrm{sgn}(x) \neq y) \leq h_\gamma(y x)$
for any $x \in \mathbb{R}$ and $y \in \{-1,1\}$, we have
\begin{eqnarray*}
	M & = & \sum_{t=1}^T \mathbf{1}(\hat y_t \neq y_t)
	  \leq
	\sum_{ \lbrace t: \hat{y}_{t} \neq y_{t} \rbrace} h_\gamma(y_t\bm{X}_{t,(i_t,j_t)}) \\
    & = &
	\mathrm{Regret}_\mathrm{OMP}(M,\mathcal{X},\bar{\bm{P}}\bar{\bm{Q}}^{T}) +
	\sum_{ \lbrace t: \hat{y}_{t} \neq y_{t} \rbrace} h_\gamma(y_t (\bar{\bm{P}}\bar{\bm{Q}}^{T})_{i_t,j_t}) \\
	& \leq &
	\mathrm{Regret}_\mathrm{OMP}(M,\mathcal{X},\bar{\bm{P}}\bar{\bm{Q}}^{T}) +
	\sum_{t=1}^T h_\gamma(y_t (\bar{\bm{P}}\bar{\bm{Q}}^{T})_{i_t,j_t}) \\
	& = &
	\mathrm{Regret}_\mathrm{OMP}(M,\mathcal{X},\bar{\bm{P}}\bar{\bm{Q}}^{T}) +
		\mathrm{hloss}(\mathcal{S},(\bm{P},\bm{Q}),\gamma),
\end{eqnarray*}
where the second equality follows from the definition of regret,
and the third equality follows from the fact that
$(\bar{\bm{P}}\bar{\bm{Q}}^{T})_{i,j} =
\bm{P}_{i}\bm{Q}_{j}^{T}/(\|\bm{P}_i\|_2 \|\bm{Q}_j\|_2)$.
Since $\bm{P}$ and $\bm{Q}$ are chosen arbitrarily,
we get (\ref{eq:bound1}).

Now, let $\bm{P}$ and $\bm{Q}$ be the matrices that attain
(\ref{equation:hloss2}). Then, the inequality above implies that
\begin{equation*}
\begin{split}
M &\leq \mathrm{Regret}_\mathrm{OMP}(M,\mathcal{X},\bar{\bm{P}}\bar{\bm{Q}}^{T}) +
		\mathrm{hloss}(\mathcal{S},\gamma) \\
&\leq	\sup_{\bm{X}^* \in \mathcal{X}} \mathrm{Regret}_\mathrm{OMP}(M, \mathcal{X}, \bm{X}^*)
		+ \mathrm{hloss}(\mathcal{S}, \gamma),
\end{split}
\end{equation*}
which proves (\ref{eq:bound2}).
\qed
\end{proof}

\subsection{Reduction from OMP with side information to generalised OSDP with bounded $\bm{\Gamma}$-trace norm}
A similar technique is used
in \cite{herbster2016mistake} and \cite{hazan2012near}. For side information matrix $\bm{M},\bm{N}$ we define a matrix $\bm{\Gamma}$ for our generalised OSDP as follows:
\begin{equation}
\bm{\Gamma}=\begin{bmatrix}
\sqrt{\mathcal{R}_{\bm{M}}\bm{M}} & 0\\
0 & \sqrt{\mathcal{R}_{\bm{N}}\bm{N}}
\end{bmatrix}.
\end{equation}

Next we define the decision class $\mathcal{K}$.
Let $N = m + n$, and for any matrices $\bm{P}$ and $\bm{Q}$
such that $\bm{P}\bm{Q}^{T} \in \mathbb{R}^{m \times n}$, we define
\[
	\bm{W}_{\bm{P},\bm{Q}} =
		\begin{bmatrix}
			\bar{\bm{P}} \\
			\bar{\bm{Q}}
		\end{bmatrix}
		\begin{bmatrix}
			\bar{\bm{P}}^{T} & \bar{\bm{Q}}^{T}
		\end{bmatrix}
	=
		\begin{bmatrix}
			\bar{\bm{P}}\bar{\bm{P}}^{T} & \bar{\bm{P}}\bar{\bm{Q}}^{T} \\
			\bar{\bm{Q}}\bar{\bm{P}}^{T} & \bar{\bm{Q}}\bar{\bm{Q}}^{T}
		\end{bmatrix}.
\]
Note that $\bm{W}_{\bm{P},\bm{Q}}$ is an $N \times N$ symmetric and
positive semi-definite matrix with its
upper right $m \times n$ component matrix
$\bar{\bm{P}}\bar{\bm{Q}}^{T}$ is a competitor matrix for the OMP problem.
So, intuitively, $\bm{W}_{\bm{P},\bm{Q}}$ can be viewed as a
positive semi-definite embedding of $\bar{\bm{P}}\bar{\bm{Q}}^{T} \in \mathcal{X}$.
Then, the decision class is any convex set $\mathcal{K} \in \mathbb{S}^{N \times N}_{++}$
that satisfies
\[
	\mathcal{K} \supseteq
		\{\bm{W}_{\bm{P},\bm{Q}} : \bm{P}\bm{Q}^{T} \in \mathbb{R}^{m \times n} \}.
\]
In this paper, we choose
\begin{equation}
\label{eq:mathcalK}
	\mathcal{K} = \{\bm{W} \in \mathbb{S}^{N \times N}_{++} : \forall i \in [n], \bm{W}_{i,i} \leq 1 \land   \mathrm{Tr}(\bm{\Gamma W \Gamma}) \leq \widehat{\mathcal{D}}   \}.
\end{equation}

Then, we define the loss matrix class $\mathcal{L}$.
For any $(i,j) \in [m] \times [n]$, let
$\bm{Z}\langle i,j\rangle \in \mathbb{S}^{N \times N}_{+}$
be a matrix such that the $(i,m+j)$-th and $(m+j,i)$-th components
are 1 and the other components are 0. More formally,
\[
	\bm{Z}\langle i,j \rangle = \frac{1}{2}
		\left(\bm{e}_i \bm{e}_{m+j}^{T} + \bm{e}_{m+j} \bm{e}_i^{T}\right),
\]
where $\bm{e}_{k}$ is the $k$-th basis vector of $\mathbb{R}^N$.
Note that when we focus on its upper right $m \times n$ component
matrix, then only the $(i,j)$-th component is 1.
Then, $\mathcal{L}$ is
\begin{equation}
\label{eq:calL}
	\mathcal{L} = \left\{
		c \bm{Z} \langle i,j \rangle
		: c \in \{-1/\gamma, 1/\gamma\}, i \in [m], j \in [n]
	\right\}.
\end{equation}

Now we are ready to describe the reduction from the OMP problem
for $\mathcal{X}$ to the OSDP problem $(\mathcal{K},\mathcal{L})$.
Let $\mathcal{A}$ be an algorithm for the OSDP problem.

Run the algorithm $\mathcal{A}$ and receive the first prediction
matrix $\bm{W}_1 \in \mathcal{K}$ from $\mathcal{A}$.

In each round $t$,
\begin{enumerate}
\item let $\bm{X}_t$ be the upper right $m \times n$ component matrix
	of $\bm{W}_t$. \\
	// $\bm{X}_{t,(i,j)} = \bm{W}_t \bullet \bm{Z} \langle i,j \rangle$
\item observe a triple $(i_t,j_t,y_t) \in [m] \times [n] \times \{-1,1\}$,
\item suffer loss $\ell_t(\bm{W}_t)$ where $\ell_t:
	\bm{W} \mapsto h_\gamma(y_t(\bm{W} \bullet \bm{Z} \langle i_t,j_t \rangle))$,
\item let $\bm{L}_t = \nabla_{\bm{W}} \ell_t(\bm{W}_{t})
	= \begin{cases}
		-\frac{y_t}{\gamma}\bm{Z} \langle i_t,j_t \rangle &
		\text{if $y_t\bm{X}_{t,(i,j)} \leq \gamma$} \\
		0 & \text{otherwise}
	\end{cases}$,
\item feed $\bm{L}_t$ to the algorithm $\mathcal{A}$ to let it proceed
	and receive $\bm{W}_{t+1}$.
\end{enumerate}

Since the loss function $\ell_t$ is convex,
a standard linearlization argument (\cite{shalev2012online})
gives
\[
	\ell_t(\bm{W}_t) - \ell_t(\bm{W}^*)
	 \leq \bm{W}_t \bullet \bm{L}_t - \bm{W}^{*} \bullet \bm{L}_t
\]
for any $\bm{W}^* \in \mathcal{K}$.
Moreover, since
$\ell_t(\bm{W}_t) = h_\gamma(y_t \bm{X}_{t,(i_t,j_t)})$ and
$\ell_t(\bm{W}_{\bm{P},\bm{Q}}) =
h_{\gamma}(y_t(\bar{\bm{P}}\bar{\bm{Q}}^{T})_{i_t,j_t})$,
the following lemma immediately follows.

\begin{lemma} \label{lem:reduction2}
Let $\mathrm{Regret}_\mathrm{OSDP}(T,\mathcal{K},\mathcal{L},\bm{W}_{\bm{P},\bm{Q}})
= \sum_{t=1}^T (\bm{W}_t - \bm{W}_{\bm{P},\bm{Q}}) \bullet \bm{L}_t$
denote the regret
of the algorithm $\mathcal{A}$ in the reduction above for a competitor
matrix $\bm{W}_{\bm{P},\bm{Q}}$
and
$\mathrm{Regret}_\mathrm{OMP}(T,\mathcal{X},\bar{\bm{P}}\bar{\bm{Q}}^{T})
= \sum_{t=1}^T (h_\gamma(y_t \bm{X}_{t,(i_t,j_t)})
- h_\gamma(y_t (\bar{\bm{P}}\bar{\bm{Q}}^{T})_{i_t,j_t})$ denote the regret
of the reduction algorithm for $\bar{\bm{P}}\bar{\bm{Q}}^{T}$.
Then,
\[
	\mathrm{Regret}_\mathrm{OMP}(T,\mathcal{X},\bar{\bm{P}}\bar{\bm{Q}}^{T})
	\leq
	\mathrm{Regret}_\mathrm{OSDP}(T,\mathcal{K},\mathcal{L},\bm{W}_{\bm{P},\bm{Q}}).
\]
\end{lemma}

Combining Lemma~\ref{lem:reduction} and Lemma~\ref{lem:reduction2},
we have the following corollary.

\begin{corollary}
\label{cor:reduction}
Assume that we have an algorithm for the OSDP problem
$(\mathcal{K},\mathcal{L})$ with
regret bound $\mathrm{Regret}_\mathrm{OSDP}(T,\mathcal{K},\mathcal{L},\bm{W}^*)$
for any $\bm{W}^{*} \in \mathcal{K}$.
Then, there exists an algorithm for the binary matrix completion
problem with the following mistake bounds.
\begin{equation*}
\begin{split}
M & \leq \inf_{\bm{P}\bm{Q}^{T} \in \mathcal{X}} (
\mathrm{Regret}_\mathrm{OSDP}(M,\mathcal{K},\mathcal{L},\bm{W}_{\bm{P},\bm{Q}})\\
&+ \mathrm{hloss}(\mathcal{S},(\bm{P},\bm{Q}), \gamma)) \\
& \leq	\sup_{\bm{W}^* \in \mathcal{K}} \mathrm{Regret}_\mathrm{OSDP}(M,\mathcal{K},\mathcal{L},\bm{W}^*)
		+ \mathrm{hloss}(\mathcal{S}, \gamma).
\end{split}
\end{equation*}
\end{corollary}

\subsection{Application to matrix completion}
For the generalised OSDP problem $(\mathcal{K},\mathcal{L})$ with bounded $\bm{\Gamma}$-trace norm defined in
(\ref{eq:mathcalK}) and (\ref{eq:calL}), where $\bm{\Gamma}$ is respect to side information matrices $\bm{M}$ and $\bm{N},$
we apply FTRL algorithm with
the generalised log-determinant regularizer.
Specifically, the FTRL algorithm makes predictions according to
the following formula:
\begin{equation}
\label{eq:FTRL}
	\bm{W}_{t+1} = \arg\min_{\bm{W} \in \mathcal{K}}
		-\ln \det(\bm{\Gamma W \Gamma} + \epsilon \bm{E})
		+ \eta \sum_{s=1}^t \bm{W} \bullet \bm{L}_{s},
\end{equation}
where $\epsilon > 0$ and $\eta > 0$ are parameters.

Moreover the following lemma shows us the quasi-dimension with respect to side information matrices $\bm{M},\bm{N} \in \mathbb{S}^{N \times N}_{++}.$ Again, $\bm{M}$ and $\bm{N}$ are identity matrices, if the side information is vacuous. In this case our generalised log-determinant regularizer becomes the regular form as $-\ln \det (\bm{W}+\epsilon \bm{E}).$
\begin{lemma}[Lemma 8 \cite{herbster2020online}]
Given side information matrices $\bm{M,N} \in \mathbb{S}^{N \times N}_{++},$ we define $\bm{\Gamma}$ as
\begin{equation}{\label{equation:side-information}}
\bm{\Gamma}=\begin{bmatrix}
			\sqrt{\mathcal{R}_{\bm{M}}}\sqrt{\bm{M}} & 0\\
			0 & \sqrt{\mathcal{R}_{\bm{N}}}\sqrt{\bm{N}}
		\end{bmatrix}.
\end{equation}
Then we obtain that
\begin{equation}
\mathrm{Tr}(\bm{\Gamma W_{P,Q}\Gamma})=\mathcal{R}_{\bm{M}}\mathrm{Tr} \left(\bar{\bm{P}}^{T} \bm{M} \bar{\bm{P}}\right)+\mathcal{R}_{\bm{N}}\mathrm{Tr} \left( \bar{\bm{Q}}^{T}\bm{N}\bar{\bm{Q}}\right).
\end{equation}
\end{lemma}

\begin{remark}
Since the definition of $\bm{\Gamma}$ in Equation (\ref{equation:side-information}), we have that $\rho=1.$
\end{remark}

Thus we set $\beta=1,$ $g = 1/\gamma,$ $\epsilon=\rho=1,$$\tau=\widehat{\mathcal{D}},$ and $\bm{\Gamma}$ is given as in Equation (\ref{equation:side-information})
next
utilise Theorem \ref{theorem:upper-bound-of-OSDP}, so we get the following result

\begin{equation}
\label{eq:regret}
	\mathrm{Regret}_\mathrm{OSDP}(T,\mathcal{K},\mathcal{L}, \bm{W}^*)
	= O \left( \frac{T\eta}{\gamma^2} + \frac{\widehat{\mathcal{D}}}{\eta} \right).
\end{equation}

Before stating our main result,
we give in Algorithm~\ref{alg:main}
the algorithm for the OMC problem with side information $\bm{M}, \bm{N}$
which is obtained by putting together the two reductions
with the FTRL algorithm (\ref{eq:FTRL}).

\begin{algorithm}
\caption{Online matrix completion with side information algorithm}
\label{alg:main}
\begin{algorithmic}[1]
\STATE Parameters: $\gamma > 0$, $\eta > 0,$ side information matrices $\bm{M} \in \mathbb{S}^{m \times m}_{++}$ and $\bm{N} \in \mathbb{S}^{n \times n}_{++}.$ Quasi dimension estimator $1 \leq \widehat{\mathcal{D}}$ $\bm{\Gamma}$ is composed as in Equation (\ref{equation:side-information}), and decision set $\mathcal{K}$ is given as (\ref{eq:mathcalK}).
\STATE Initialize $ \forall \bm{W} \in \mathcal{K},$ setting $\bm{W}_{1} =\bm{W}$.
\FOR {$t=1,2,\dots,T$}
\STATE Receive $(i_t,j_t) \in [m] \times [n]$.
\STATE Let $\bm{Z}_t = \frac{1}{2}(
	\bm{e}_{i_t}\bm{e}_{m+j_t}^{T} +
	\bm{e}_{m+j_t}\bm{e}_{i_t}^{T})$.
\STATE Predict $\hat{y}_t = \mathrm{sgn}(\bm{W}_t \bullet \bm{Z}_t)$ and
 receive $y_t \in \{-1,1\}$.
\IF{$\hat y_t \neq y_t$}
\STATE Let $\bm{L}_t = \frac{-y_t}{\gamma} \bm{Z}_t$ and
$\bm{W}_{t+1} = \arg\min_{\bm{W} \in\mathcal{K}} -\ln \det(\bm{\Gamma W \Gamma} + \bm{E})
	+ \eta \sum_{s=1}^t \bm{W} \bullet \bm{L}_s$.
\ELSE
\STATE Let $\bm{L}_t = 0$ and $\bm{W}_{t+1} = \bm{W}_{t}$.
\ENDIF
\ENDFOR
\end{algorithmic}
\end{algorithm}

Usually we set $\eta = \sqrt{\gamma^2\widehat{\mathcal{D}}/T}$ to minimize
(\ref{eq:regret}), we obtain $O\left(\sqrt{\widehat{\mathcal{D}}T/\gamma^2}\right)$ regret bound.
But in our case, the horizon $T$ is set to be the number of mistakes $M$
through the reduction, which is unknown in advance.
Nevertheless, the next theorem shows that we can choose $\eta$
independent of $M$ to derive a good mistake bound.

\begin{theorem}
\label{th:main}
Algorithm~\ref{alg:main} with $\eta = c \gamma^2$ for some
$c > 0$ achieves
\begin{equation}
\label{eq:mbound}
	M = \sum_{t=1}^T \mathbb{I}_{\hat{y}_{t} \neq y_{t}}
	= O\left( \frac{\widehat{\mathcal{D}}}{\gamma^2} \right) + 2 \mathrm{hloss}(\mathcal{S}, \gamma).
\end{equation}
\end{theorem}

\begin{proof}
Combining Corollary~\ref{cor:reduction} and
the regret bound (\ref{eq:regret}), we have
\[
	M = O \left(
		\frac{M \eta}{\gamma^2} + \frac{\widehat{\mathcal{D}}}{\eta}
	\right) + \mathrm{hloss}(\mathcal{S},\gamma).
\]
Choosing $\eta = c \gamma^2$ for sufficiently small constant $c$,
we get
\[
	M \leq \frac{M}{2} + O\left( \frac{\widehat{\mathcal{D}}}{\gamma^2} \right)
		+ \mathrm{hloss}(\mathcal{S},\gamma),
\]
from which (\ref{eq:mbound}) follows.
\qed
\end{proof}

Again if the side information is vacuous, which means that $\bm{M}, \bm{N}$ are identity matrices, from Remark \ref{rem:original-case} and Theorem \ref{th:main}, we can set that $\widehat{\mathcal{D}}=m+n$ and obtain the mistake bound as follows:
$$O\left( \frac{m+n}{\gamma^2} + 2 \mathrm{hloss}_{\bm{P}\bm{Q}^{T} \in \mathbb{R}^{m \times n}}(\mathcal{S},(\bm{P},\bm{Q}), \gamma)\right).$$

In contrast, there is a case where side information matters non-trivially. Especially, if $\bm{U}$ contains $(k \times l)$-biclustered structure(the details are in Appendix) then we obtain that $\widehat{\mathcal{D}} \in O(k+l),$ which is strictly smaller than $O(m+n).$

Note that in realizable case, our mistake bound becomes $O\left( \frac{\widehat{\mathcal{D}}}{\gamma^2} \right),$ which improves the previous bound $O\left(\frac{\widehat{\mathcal{D}}}{\gamma^{2}} \ln (m+n)\right)$ in \cite{herbster2020online}, removing the logarithmic factor $\ln(m+n).$ Furthermore, this bound matches the previously known lower bound of Herbster et.al. \cite{herbster2016mistake}. When $\bm{U}$ contains $(k \times l)$-biclustered structure ($k\geq l$), $\gamma$ can be set as $\gamma=\frac{1}{\sqrt{l}}$ and our regret bound becomes $O(l m).$ On the other hand, the lower bound of Herbster et.al. is $\Omega( lm).$ Thus, the mistake bound of Theorem \ref{th:main} is optimal.

\section{Application to online similarity prediction with side information}
In this section, we show that our reduction method and generalised log-determinant regularizer work in online similarity prediction with side information.

Let $G=(V,E)$ be an undirected graph with $n=|V|$ vertices and $m=|E|$ edges. Assign vertices to $K$ classes such that $\lbrace y_{1},\cdots, y_{n}\rbrace$ where $y_{i} \in \lbrace 1,\cdots, K\rbrace.$ On each round $t,$ for a given pair of vertices $(i_{t},j_{t})$ algorithm needs to predict whether they are in the same class denoted as $\hat{y}_{i_{t},j_{t}}.$ If they are in the same class then $y_{i_{t},j_{t}}=1,$ $y_{i_{t},j_{t}}=-1,$ otherwise. Our target is to give a bound of the prediction mistakes $M=\sum_{t=1}^{T}\mathbb{I}_{\hat{y}_{i_{t},j_{t}}\neq y_{i_{t},j_{t}}}.$

\begin{definition}
The set of cut-edges in $(G,y)$ is denoted as $\Phi^{G}(y)=\lbrace (i,j) \in E: y_{i}\neq y_{j}\rbrace$ we abbreviate it to $\Phi^{G}$ and the cut-size is given as $|\Phi^{G}(y)|.$ The set of cut-edges with respect to class label $k$ is denoted as $\Phi_{s}^{G}(y)=\lbrace (i,j) \in E: s\in \lbrace y_{i},y_{j}\rbrace, y_{i} \neq y_{j}\rbrace.$ Note that $\sum_{s=1}^{k}|\Phi^{G}_{s}(y)|=2|\Phi^{G}(y)|.$ Given $\bm{A} \in \mathbb{R}^{n \times n}$ such that $A_{ij}=A_{ji}=1$ if $(i,j) \in E(G)$ and $A_{ij}=0$ otherwise. $\bm{D}$ is denoted as diagonal matrix with $\bm{D}_{ii}$ is the degree of vertex $i.$ We define the Laplacian as $\bm{L}=\bm{D}-\bm{A}.$
\end{definition}

\begin{definition}
If $G$ is identified with a resistive network such that each edge is a unit resistor, then the effective resistance $R^{G}_{i,j}$ between pair $(i,j) \in V^{2}$ can be defined as $R^{G}_{i,j}=(e_{i}-e_{j})\bm{L}^{+}(e_{i}-e_{j}),$ where $e_{i}$ is the $i$-th vector in the canonical basis of $\mathbb{R}^{n}.$
\end{definition}

\cite{gentile2013online} gave a mistake bound as $M \leq O\left(\frac{|\Phi^{G}|\max_{(i,j) \in V^{2}}R^{G}_{i,j}}{\gamma^{2}} \ln n\right).$

If we utilise our reduction method, same as in previous main part here we denote that $\bar{\bm{P}},\bar{\bm{Q}} \in \mathcal{B}^{n,k}$ and $\bm{P}\bm{Q}^{T}=\gamma \bm{U},$ where $\bm{U}$ is the potential online matrix for similarity prediction. Therefore we give the decision set as
\begin{equation}
\bm{W}_{\bm{P},\bm{Q}} =
		\begin{bmatrix}
			\bar{\bm{P}} \\
			\bar{\bm{Q}}
		\end{bmatrix}
		\begin{bmatrix}
			\bar{\bm{P}}^{T} & \bar{\bm{Q}}^{T}
		\end{bmatrix}
	=
		\begin{bmatrix}
			\bar{\bm{P}}\bar{\bm{P}}^{T} & \bar{\bm{P}}\bar{\bm{Q}}^{T} \\
			\bar{\bm{Q}}\bar{\bm{P}}^{T} & \bar{\bm{Q}}\bar{\bm{Q}}^{T}
		\end{bmatrix}.
\end{equation}

Side information is given as PD-Laplacian $\bar{\bm{L}}$ from Laplacian $\bm{L}$ of graph $G,$ thus we have
\begin{equation}
\bm{\Gamma} =
\begin{bmatrix}
\sqrt{\mathcal{R}_{\bar{\bm{L}}}\bar{\bm{L}}} & 0\\
0 & \sqrt{\mathcal{R}_{\bar{\bm{L}}}\bar{\bm{L}}}
\end{bmatrix}
\end{equation}

Meanwhile given sparse matrix $\bm{Z}_{t}$ in following equation
\begin{equation}
\bm{Z}_{t}=\frac{1}{2}(e_{i}e^{T}_{n+j}+e_{n+j}e^{T}_{i}),
\end{equation}
and $c \in \lbrace -1/\gamma, 1/\gamma \rbrace.$

Hence we can give the reduced generalised OSDP problem $(\mathcal{K},\mathcal{L})$ with bounded $\bm{\Gamma}$-trace norm as follows:
\begin{equation*}
\begin{split}
&\mathcal{K}=\left\{\bm{X} \in \mathbb{S}_{++}^{n \times n}: |\bm{X}_{ii}| \leq 1, \mathrm{Tr}(\bm{\Gamma  X\Gamma}) \leq \widehat{\mathcal{D}}\right\} \\
&\mathcal{L}=\left\{c\bm{Z}\langle i,j \rangle: : c \in \{-1/\gamma, 1/\gamma\}, i \in [n], j \in [n] \right\},
\end{split}
\end{equation*}
where $\bm{\Gamma}$ is defined as above.

According to \cite{herbster2020online} if
$\bm{U}$ obtains the $(k,k)$-biclustered structure, and $\bm{U}=\bm{RU^{*}R^{T}},$ we have that
$\mathrm{Tr}(\bm{\Gamma X \Gamma}) \leq \mathrm{Tr}(\mathcal{R}_{\bar{\bm{L}}}(\bm{R^{T}U^{*}R})) \leq 2\mathrm{Tr}(\bm{R}^{T}\bm{LR})\mathcal{R}_{\bm{L}}+2k \leq O(k).$ Moreover we have that
\begin{equation}
\begin{split}
M & \leq O\left( \frac{\mathrm{Tr}(\bm{R}^{T}\bm{LR})\mathcal{R}_{\bm{L}}}{\gamma^{2}}  \right)+\min_{\bar{\bm{P}}\bar{\bm{Q}}^{T} \in \mathcal{X}}\sum_{t=1}^{T}h_{\gamma}(y_{i_{t},j_{t}}(\bm{P}\bm{Q}^{T})_{i_{t},j_{t}})\\
&\leq O\left( \frac{k}{\gamma^{2}}  \right)+\min_{\bar{\bm{P}}\bar{\bm{Q}}^{T} \in \mathcal{X}}\sum_{t=1}^{T}h_{\gamma}(y_{i_{t},j_{t}}(\bm{P}\bm{Q}^{T})_{i_{t},j_{t}}),
\end{split}
\end{equation}
where $\mathcal{X}=\lbrace \bar{\bm{P}}\bar{\bm{Q}}^{T}: \bm{PQ}^{T} \in \mathbb{R}^{n \times n}: \mathcal{R}_{\bar{\bm{L}}}\mathrm{Tr}(\bar{\bm{P}}^{T}\bar{\bm{L}}\bar{\bm{P}})+\mathcal{R}_{\bar{\bm{L}}}\mathrm{Tr}(\bar{\bm{Q}}^{T}\bar{\bm{L}}\bar{\bm{Q}}) \leq \widehat{\mathcal{D}} \rbrace$ for some $\widehat{\mathcal{D}} \geq \mathcal{D}_{\bar{\bm{L}}, \bar{\bm{L}}}^{\gamma}(\bm{U}).$

\begin{remark}
According to \cite{herbster2020online}, we have that $\mathrm{Tr}(\mathcal{R}_{\bar{\bm{L}}}(\bm{R^{T}U^{*}R}))\leq 2\sum_{i,j}\Vert R_{i}-R_{j}\Vert^{2}+2k,$ where $\sum_{(i,j) \in E}\Vert R_{i}-R_{j}\Vert^{2}$ counts only when there is a edge between different classes. Due to the definition of $|\Phi^{G}|,$ we have that $\sum_{(i,j)\in E}\Vert R_{i}-R_{j} \Vert^{2}=|\Phi^{G}|.$ On the other hand, $\mathcal{R}_{\bm{L}}=\max_{ii}\bm{L}^{+}$ so we obtain that $\mathcal{R}_{\bm{L}} \geq e_{i}^{T}\bm{L}^{+}e_{i}, \forall i\in [k].$ It implies that $2\mathcal{R}_{\bm{L}} \geq \max_{(i,j) \in V^{2}}R^{G}_{i,j}.$
\end{remark}

\section{Conclusion}
In this paper, on the one hand we define a generalised OSDP problem with bounded $\bm{\Gamma}$-trace norm. To solve this problem, we involve FTRL with generalised log-determinant regularizer and achieve regret bound as $O((1+\rho)g\sqrt{\beta \tau T}).$ On the other hand, we utilise our result to OMC with side information particularly. We reduce OMC with side information to our new OSDP with bounded $\bm{\Gamma}$-trace norm, and obtain a tighter mistake bound than previous work by removing logarithmic factor.

\bibliographystyle{splncs04}
\bibliography{ref}
\section{Appendix}
\begin{lemma}\cite{hazan201210}\label{lemma:lemma-for-upper-bound-of-strongly-convex}
Let $R: \mathcal{K}\rightarrow \mathbb{R}$ be $s$-strongly convex with respect to $\mathcal{L}$ for $\mathcal{K}.$ Then the FTRL with the regularizer $R$ applied to $(\mathcal{K},\mathcal{L})$ achieves
\begin{equation}
\mathrm{Regret_{OSDP}}(T,\mathcal{K},\mathcal{L},\bm{W}^{*})  \leq \frac{H_{0}}{\eta}+\frac{\eta}{s}T,
\end{equation}
where $H_{0}=\max_{\bm{W},\bm{W}^{'} \in \mathcal{K}}(R(\bm{W})-R(\bm{W}^{'})),$ $\bm{W}^{*}$ is the static optimal solution of $\sum_{t=1}^{T}\bm{L}_{t} \bullet \bm{W}.$
In particular if we choose $\eta=\sqrt{sH_{0}/T}$ then we have
\begin{equation}
\mathrm{Regret_{OSDP}}(T,\mathcal{K},\mathcal{L},\bm{W}^{*})  \leq 2\sqrt{\frac{H_{0}T}{s}}.
\end{equation}
\end{lemma}

\begin{lemma}[Lemma A.1 \cite{moridomi2018online}]\label{lemma:A-1}
Let $P$ and $Q$ be probability distribution over $\mathbb{R}^{V}$ and $\phi_{P}(u)$ and $\phi_{Q}(u)$ be their characteristic functions, respectively. Then
\begin{equation}
\max_{u \in \mathbb{R}^{N}}|\phi_{P}(u)-\phi_{Q}(u)| \leq \int_{x} |P(x)-Q(x)| dx,
\end{equation}
the right hand side is the total variation distance between any distribution $Q$ and $P.$
\end{lemma}

\begin{lemma}[Lemma A.2 \cite{christiano2014online}]\label{lemma:A-2}
Let $P$ and $Q$ be probability distributions over $\mathbb{R}^{N}$ with total variation distance $\delta.$ Then
\begin{equation}
H(\alpha P+(1-\alpha)Q) \leq \alpha H(P)+(1-\alpha)H(Q)-\alpha(1-\alpha)\delta^{2},
\end{equation}
where $H(P)=\mathbb{E}_{x \sim P}[\ln P(x)].$
\end{lemma}

\begin{lemma}[Lemma A.3 \cite{moridomi2018online}]\label{lemma:A-3}
For any probability distribution $P$ over $\mathbb{R}^{V}$ with zero mean and covariance matrix $\Sigma$ its entropy is bounded by the log-determinant of covariance matrix. That is
\begin{equation}
-H(P) \leq \frac{1}{2}\ln (\det (\Sigma)(2\pi e)^{V}).
\end{equation}
\end{lemma}

\begin{lemma}[Lemma A.4 \cite{moridomi2018online}]\label{lemma:A-4}
\begin{equation}
e^{\frac{-x}{2}}- e^{-\frac{1-x}{2}} \geq \frac{e^{-1/4}}{2}(1-2x),
\end{equation}
for $0 \leq x 1/2.$
\end{lemma}

\begin{lemma}[Lemma 5.4 \cite{moridomi2018online}]\label{lemma:lemma-for-strongly-convex}
Let $\bm{X}, \bm{Y} \in \mathbb{S}_{++}^{N \times N}$ be such that for all $i \in [N]$ $  |\bm{X}_{i,i}|\leq \beta^{'}$ and $|\bm{Y}_{i,i}| \leq \beta^{'}$ Then for any $\bm{L} \in \mathcal{L}=\lbrace \bm{L} \in \mathbb{S}^{N \times N}_{+}: \Vert \mathrm{vec}(\bm{L}) \Vert_{1}\leq g \rbrace$ there exists that
\begin{equation}
|\bm{X}_{i,j}-\bm{Y}_{i,j}| \geq \frac{|\bm{L} \bullet (\bm{X}-\bm{Y})|}{4\beta^{'}g}(\bm{X}_{i,i}+\bm{Y}_{i,i}+\bm{X}_{j,j}+\bm{Y}_{j,j}).
\end{equation}
\end{lemma}

\section{Appendix B. Biclustered structure and ideal case of quasi-dimension}
As in \cite{herbster2020online} we define $(k,l)$-biclustered structure as follows: For $m \geq k$ and $n \geq l,$
\begin{definition}
the class of $(k,l)$-binary biclustered matrices is defined as
\begin{equation*}
\begin{split}
\mathbb{B}^{m\times n}_{k,l}&=\lbrace U\in \lbrace -1,+1\rbrace^{m\times n}: \bm{r} \in [k]^{m},\bm{c}\in [l]^{n},V \in \lbrace 1,-1\rbrace^{k \times l},\\
 &U_{i,j}=V_{r_{i},c_{j}}, i \in [m], j \in [n]\rbrace.
\end{split}
\end{equation*}
\end{definition}

Denote $\mathcal{B}^{m,d}=\lbrace \bm{R} \subset \lbrace 0,1 \rbrace^{m \times d}: \Vert \bm{R}_{i} \Vert_{2}=1, i \in [m], \mathrm{rank}(\bm{R})=d \rbrace,$ for any matrix $\bm{U} \in \mathbb{B}^{m,n}_{k,l}$ we can decompose $\bm{U}=\bm{RU^{*}}\bm{C}^{T}$ for some $\bm{U}^{*} \in \lbrace -1,+1 \rbrace^{k \times l},\bm{R} \in \mathcal{R}^{m,k}$ and $\bm{C} \in \mathcal{B}^{n,l}.$ In \cite{herbster2020online} if the comparator matrix $\bm{U} \in \mathbb{B}^{m \times n}_{k,l},$ we know that $\mathcal{D}_{\bm{M},\bm{N}}^{\gamma} \leq 2\mathrm{Tr}(\bm{R}^{T}\bm{MR})\mathcal{R}_{\bm{M}}+2\mathrm{Tr}(\bm{C}^{T}\bm{NC})\mathcal{R}_{\bm{N}}+2k+2l,$ if $\bm{M,N}$ are PD-Laplacian and $\mathcal{D}_{\bm{M},\bm{N}}^{\gamma} \leq O(k+l).$

In graph-based semi-supervised learning \cite{herbster2020online}, for a given row corresponding to a vertex in the row graph. The weight of edge $(i,j)$ represents our prior belief that row $i$ and row $j$ share the same underlying factor. Hence we may build a graph based on vectorial data associated with the rows, for example, user demographics. Assume that we know the partition of $[m]$ vertices that maps rows to $k$ factors. The rows that share factors have an edge between them and there are no other edges. Therefore we have a graph with $k$ disjoint cliques. We assume that this graph has a structure that any pair of vertices in this graph can be connected with a path within 4 length. Choose the side information matrix $\bm{M}$ as PD-Laplacian $\bar{\bm{L}}$ of this row graph. For columns in a partition of $[n]$ with $l$ cliques, we do the same work. At last we obtain that the $\mathcal{D}_{\bm{M},\bm{N}}^{\gamma} \in O(k+l).$


\end{document}